\documentclass[preprint,12pt]{elsarticle}

\usepackage{graphicx}
\usepackage{amssymb}
\usepackage{amsthm} 
\usepackage{amsmath}
\usepackage{breqn}
\usepackage[frozencache] {minted}
\usepackage[capitalize]{cleveref}
\usepackage{url}
\usepackage{bibentry}
\usepackage{lmodern}
\nobibliography*

\newtheorem{thm}{Theorem}[section]
\newtheorem{cor}{Corollary}[section]
\newtheorem{prop}{Proposition}[section]

\newcommand{\ser}{\mathbin{\bowtie}}
\newcommand{\pll}{\mathbin{\|}}

\usepackage{lineno}

\usepackage{tikz}
\usetikzlibrary{calc}
\usetikzlibrary{decorations.pathmorphing}
\usepackage{pgf}

\journal{{\tt arXiv.org}}

\begin{document}

\begin{frontmatter}

%% Title, authors and addresses

\title{Chromatic roots at $2$ and at the Beraha number $B_{10}$}
%\title{Real chromatic roots at }

%% use the tnoteref command within \title for footnotes;
%% use the tnotetext command for the associated footnote;
%% use the fnref command within \author or \address for footnotes;
%% use the fntext command for the associated footnote;
%% use the corref command within \author for corresponding author footnotes;
%% use the cortext command for the associated footnote;
%% use the ead command for the email address,
%% and the form \ead[url] for the home page:
%%
%% \title{Title\tnoteref{label1}}
%% \tnotetext[label1]{}
%% \author{Name\corref{cor1}\fnref{label2}}
%% \ead{email address}
%% \ead[url]{home page}
%% \fntext[label2]{}
%% \cortext[cor1]{}
%% \address{Address\fnref{label3}}
%% \fntext[label3]{}

%% use optional labels to link authors explicitly to addresses:
%% \author[label1,label2]{<author name>}
%% \address[label1]{<address>}
%% \address[label2]{<address>}

\author{Daniel J. Harvey}
\ead{daniel.harvey87@gmail.com}
\author{Gordon F. Royle}
\ead{gordon.royle@uwa.edu.au}

\address{Department of Mathematics and Statistics, University of Western Australia}

\begin{abstract}
By the construction of suitable graphs and the determination of their chromatic polynomials, we resolve two open questions concerning real chromatic roots. First we exhibit graphs for which the Beraha number $B_{10} = (5 + \sqrt{5})/2$ is a chromatic root. As it was previously known that no other non-integer Beraha number is a chromatic root, this completes the determination of precisely which Beraha numbers can be chromatic roots. Next we construct an infinite family of $3$-connected graphs such that for any $k \geqslant 1$, there is a member of the family with $q=2$ as a chromatic root of multiplicity at least $k$. The former resolves a question of Salas and Sokal [J. Statist. Pys. 104 (2001) pp. 609--699] and the latter a question of Dong and Koh [J. Graph Theory 70 (2012) pp. 262--283].
\end{abstract}

\begin{keyword}
Chromatic Polynomial \sep Chromatic Root \sep Beraha Number
%% keywords here, in the form: keyword \sep keyword

%% MSC codes here, in the form: \MSC code \sep code
%% or \MSC[2008] code \sep code (2000 is the default)

\end{keyword}

\end{frontmatter}

%%
%% Start line numbering here if you want
%%
%\linenumbers

%% main text
\section{Introduction}\label{intro}

Given a graph $G$, its \emph{chromatic polynomial} $P_G$ is the function defined by the property that $P_G(q)$ is the number of proper $q$-colourings of $G$ whenever $q$ is a non-negative integer. As its name suggests, $P_G$ is a polynomial in $q$ and so it is possible to evaluate the function at arbitrary real and complex values of $q$, regardless of whether these evaluations have any combinatorial significance. In particular it is possible to find the roots of the polynomial, either real or complex, and there is an extensive body of work relating graph-theoretical properties to the location of these \emph{chromatic roots}. The fundamental connection between complex chromatic roots and phase transitions in the $q$-state Potts model means that many of these results appear in the statistical physics literature. Sokal's \cite{MR2187739} survey paper is a comprehensive introduction and overview to this body of work. 

The fundamental results regarding which real numbers are chromatic roots (of any graph) can be summarised in a 
single sentence: There are no chromatic roots in the real intervals $(-\infty,0)$, $(0,1)$ or $(1,32/27]$ (Jackson \cite{MR1264037}), while chromatic roots are dense in the interval $(32/27,\infty)$ (Thomassen \cite{MR1483433}). There are more fine-grained results for specific classes of graphs, such as planar graphs, but we do not need these now.

The \emph{Beraha numbers} are the infinite sequence of numbers $\{B_n\}_{n=1}^{\infty}$ where
\begin{equation}
B_n = 2 + 2 \cos \left( \frac{2 \pi}{n} \right),
\end{equation}
with the sequence starting 
\[
B_1 = 4,\ 0,\ 1,\ 2, \varphi+1,\ 3,\  3.2469796,\ 2 + \sqrt{2},\ 3.5320889,\ \varphi+2,\ \ldots
\]
where the values for $B_7$ and $B_9$ are rounded to 8 significant figures, and $\varphi = (1+\sqrt{5})/2$ is the Golden Ratio. The sequence was first identified by Beraha in the course of studying chromatic roots of planar graphs, and cropped up again as  limit points of complex chromatic roots of the various lattices studied by statistical physicists. His conjecture that for each Beraha number $B_\ell$, there are planar triangulations with chromatic roots arbitrarily close to $B_\ell$ is still not resolved.

One of the more remarkable facts involving some of the Beraha numbers is Tutte's famous ``Golden Identity'' \cite{MR0272676} relating the values of the chromatic polynomial of a planar triangulation at $B_{5}$ and $B_{10}$. The Golden Identity asserts that if $T$ is a planar triangulation on $n$ vertices, then
\begin{equation}
P_T(\varphi+2) = (\varphi+2) \varphi^{3n-10} P_T(\varphi+1)^2,
\end{equation}
with the implication that for any planar triangulation $T$, either both or neither $B_5$ and $B_{10}$ are chromatic roots of $T$. In 1970, Tutte already knew that  $P_T(B_5) \not= 0$ for any planar triangulation $T$ and therefore that $P_T(B_{10})$ is not only non-zero, but strictly positive. Recently Perrett and Thomassen \cite{perrett_thomassen_2018} have shown that the same conclusion holds for all planar graphs, not just triangulations.  

Salas and Sokal \cite{MR1853428} observed that not only is $B_5$ not the chromatic root of any planar triangulation, but in fact it is not the chromatic
root of \emph{any graph} whatsoever. This follows because the minimal polynomial of $B_5$ is $q^2 - 3q +1$ and so any integer polynomial with $B_5 = (3 + \sqrt{5})/2$ as a root must also have its algebraic conjugate (the other root of the minimal polynomial) $B_5^* = (3 - \sqrt{5})/2$ as a root. As $B_5^* \approx 0.38196601$, this lies in the forbidden interval $(0,1)$ and so it follows that neither $B_5$ nor
$B_5^*$ are chromatic roots of any graph. In a similar fashion, they showed that no non-integer Beraha number is the chromatic root of any graph, with the possible exception of $B_{10} = \varphi+2$. The minimal polynomial of $B_{10}$ is $q^2 - 5 q + 5$ and so the algebraic conjugate $B_{10}^* \approx 1.381966011$ which is not \emph{a priori} forbidden.  They concluded that \emph{``The exceptional case $n=10$ is very curious''}.

In \cref{beraha} we resolve this exceptional case by exhibiting graphs with $B_{10}$ as a chromatic root, thereby completing the determination of exactly which Beraha numbers can be chromatic roots.

\begin{prop}
The only Beraha numbers that are chromatic roots are the integer Beraha numbers $\{B_1, B_2, B_3, B_4, B_6\}$ and $B_{10}$.
\end{prop}

The second part of this paper is concerned with the multiplicity of the integer $2$ as a chromatic root. All expository introductions to the theory of chromatic polynomials make the following observations:
\begin{itemize}
\setlength{\itemsep}{0pt}
\item If $G$ is connected, then $q=0$ is a simple root of $P_G(q)$.
\item If $G$ is $2$-connected, then $q=1$ is a simple root of $P_G(q)$.
\end{itemize}
However the analagous statement for $3$-connected graphs is just not true --- there are $3$-connected graphs with either no chromatic roots at $q=2$ (bipartite graphs) or a double chromatic root at $q=2$ (Dong \\ Koh \cite{MR2946075}). Despite this, there are specific families of graphs (for example, planar triangulations) whose $3$-connected members are guaranteed to have a simple chromatic root at $q=2$. This suggests that it may be possible to recover a statement of the form
\begin{itemize}
\item If $G$ is a $3$-connected graph and has
<<\emph{some graphical property}>>, then $q=2$ is a simple root of $P_G(q)$.  
\end{itemize}
As a bipartite graph does not have $q=2$ as a chromatic root at all, the graphical property should either explicitly or (preferably) implicitly exclude bipartite graphs. 

Dong and Koh \cite{MR2946075} addressed this question, finding an infinite family $\mathcal{J}$ of $3$-connected graphs such that $(q-2)^2 \mid P_G(q)$ for all $G \in \mathcal{J}$. They exhibited a $13$-vertex graph $H$ (attributed to Jackson) which has a double chromatic root at $q=2$ (see \cref{herschel}). They observed that the clique sum (aka ``gluing'') of $k$ copies of $H$ across its unique triangle yields a graph $H_k$ that has $2$ as a chromatic root multiplicity $k+1$. However, in the study of chromatic polynomials, graphs obtained by clique sums are in some sense trivial or uninteresting cases. A graph is the clique sum of two smaller graphs if and only if it has a \emph{complete cutset}. Therefore Dong and Koh asked whether it was possible to find a $3$-connected graph $G$, without a complete cutset, such that $(q-2)^3 \mid P_G(q)$, and, more generally, with a chromatic root at $2$ of arbitrarily high multiplicity.

\begin{figure}
\begin{center}
\begin{tikzpicture}
\tikzstyle{vertex}=[circle, draw=black, fill=black!25!white, inner sep = 0.55mm]
\node [vertex] (v0) at (0,0) {};
\node [vertex] (v1) at (0,4) {};
\node [vertex] (v2) at (1,2) {};
\node [vertex] (v3) at (1,3) {};
\node [vertex] (v4) at (2,1) {};
\node [vertex] (v5) at (2,2) {};
\node [vertex] (v6) at (2,3) {};
\node [vertex] (v7) at (3,1) {};
\node [vertex] (v8) at (3,2) {};
\node [vertex] (v9) at (4,4) {};
\node [vertex] (va) at (3.5,0) {};
\node [vertex] (vb) at (4,0.5) {};
\node [vertex] (vc) at (3.5,0.5) {};
\draw (v0)--(v1);
\draw (v0)--(v2);
\draw (v0)--(v4);
\draw (v1)--(v3);
\draw (v1)--(v9);
\draw (v2)--(v3);
\draw (v2)--(v5);
\draw (v4)--(v5);
\draw (v4)--(v7);
\draw (v3)--(v6);
\draw (v5)--(v6);
\draw (v6)--(v9);
\draw (v7)--(v8);
\draw (v0)--(va);
\draw (v7)--(vc);
\draw (v9)--(vb);
\draw (v8)--(v9);
\draw (v5)--(v8);
\draw (va)--(vb)--(vc)--(va);
\end{tikzpicture}
\caption{A graph $H$ such that $(q-2)^2 \mid P_H(q)$}
\label{herschel}
\end{center}
\end{figure}
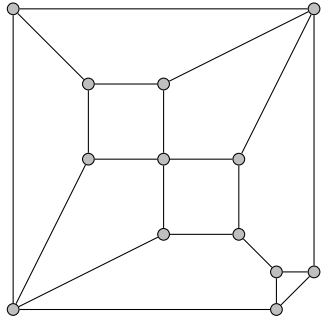

In \cref{multiroot} we resolve their question by constructing a family of graphs such that for all $k \geqslant 3$, there is a unique member of this family that is $3$-connected, has no complete cutset, and has $q=2$ as a chromatic root of multiplicity at least $k$. The family is based on replacing the rim vertices in the $k$-spoke wheel $W_k$ with copies of $K_{3,3}$ in a particular way. This yields a \emph{recursive family} of graphs in the sense of Biggs, Damerell and Sands \cite{MR0294172}. Using techniques developed by  Noy and Rib\'o \cite{MR2037635}, we find the linear recurrence satisfied by the chromatic polynomials of this family, and use this to show that $(q-2)^k$ divides the chromatic polynomial of the $k$-th member of the family.

\begin{prop}
For all natural numbers $k$, there is a $3$-connected graph $G$ such that $(q-2)^k \mid P_G(q)$.  
\end{prop}

In an Appendix, we give some {\tt SageMath} code to construct the family of graphs and compute their chromatic polynomials according to the linear recurrence; this is to assist any reader wishing to verify our claims.

%In a concluding section, we discuss the closely-related question of trying to determine when a $3$-connected graph has no chromatic roots in the interval $(1,2)$. Both questions seem to be intimately related to the presence or absence of certain \emph{stable cutsets} (that is, cutsets that induce an empty graph), but an exact characterization of graphs with no chromatic roots in $(1,2)$ remains elusive.

\section{Graphs with $B_{10}$ as a chromatic root}
\label{beraha}

Two graphs with $B_{10}$ as a chromatic root are shown in Figure~\ref{smallest}. Each example is obtained by replacing a single edge of a suitable complete graph (either $K_6$ or $K_5$) with a small ``gadget''. Each of the gadgets is a small $2$-terminal series-parallel graph (see  Royle and Sokal \cite{MR3416855} for an extensive discussion of $2$-terminal series-parallel graphs). The factored chromatic polynomials of these graphs are 

{\small 
\begin{align*}
P_{G_1}(q) &= q (q-1) (q-2) (q-3)^2 (q-4) \underbrace{\left(q^2-5 q+5\right)}_{\text{Min. pol of $B_{10}$}} \left(q^3-4 q^2+8 q-7\right)\\
P_{G_2}(q) &= q (q-1) (q-2) (q-3)  \underbrace{\left(q^2-5 q+5\right)}_{\text{Min. pol of $B_{10}$}}\left(q^5-8 q^4+30 q^3-63 q^2+73 q-36\right),
\end{align*}}
with the factor $q^2-5q+5$ highlighted.

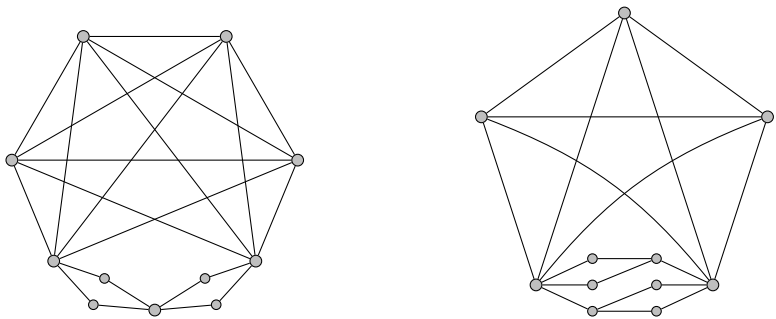
\begin{figure}
\begin{center}
\begin{tikzpicture}[rotate=90,scale=0.95]
\tikzstyle{vertex}=[circle, draw=black, fill=black!25!white, inner sep = 0.55mm]
\tikzstyle{svertex}=[circle, draw=black, fill=black!25!white, inner sep = 0.45mm]
\node [vertex] (v0) at (30:2cm) {};
\node [vertex] (v1) at (90:2cm) {};
\node [vertex] (v2) at (135:2cm) {};
\node [vertex] (v3) at (225:2cm) {};
\node [vertex] (v4) at (270:2cm) {};
\node [vertex] (v5) at (330:2cm) {};

\node [vertex] (v6) at (180:2.1cm) {};

\node [svertex] (v7) at (157:1.8cm) {};
\node [svertex] (v8) at (157:2.2cm) {};

\node [svertex] (v9) at (203:1.8cm) {};
\node [svertex] (v10) at (203:2.2cm) {};

\draw (v0)--(v1)--(v2)--(v0);
\draw (v0)--(v3);
\draw (v0)--(v4);
\draw (v0)--(v5);
\draw (v3)--(v4)--(v5)--(v3);
\draw (v1)--(v3);
\draw (v1)--(v4);
\draw (v1)--(v5);
\draw (v2)--(v4);
\draw (v2)--(v5);

\draw (v2)--(v7)--(v6)--(v8)--(v2);
\draw (v3)--(v9)--(v6)--(v10)--(v3);

\end{tikzpicture}
\hspace{2cm}
\begin{tikzpicture}
\tikzstyle{vertex}=[circle, draw=black, fill=black!25!white, inner sep = 0.55mm]
\tikzstyle{svertex}=[circle, draw=black, fill=black!25!white, inner sep = 0.45mm]
\node [vertex] (v2) at (90:2cm) {};
\node [vertex] (v3) at (162:2cm) {};
\node [vertex] (v4) at (234:2cm) {};
\node [vertex] (v0) at (306:2cm) {};
\node [vertex] (v1) at (18:2cm) {};

\node [svertex] (v5) at ($(v4) + (0.75,0.35)$) {};
\node [svertex] (v6) at ($(v4) + (0.75,0)$) {};
\node [svertex] (v7) at ($(v4) + (0.75,-0.35)$) {};

\node [svertex] (v8) at ($(v0) - (0.75,0.35)$) {};
\node [svertex] (v9) at ($(v0) - (0.75,0)$) {};
\node [svertex] (v10) at ($(v0) - (0.75,-0.35)$) {};

\draw (v0)--(v1);
\draw (v0)--(v2);
\draw [bend left = 15] (v4) to (v1);
\draw [bend left = 15] (v3) to (v0);
\draw (v4)--(v2);
\draw (v4)--(v3);

\draw (v1)--(v2)--(v3)--(v1);

\draw (v4)--(v5);
\draw (v4)--(v6);
\draw (v4)--(v7);

\draw (v0)--(v8);
\draw (v0)--(v9);
\draw (v0)--(v10);

\draw (v10)--(v5);
\draw (v10)--(v6);

\draw (v7)--(v8);
\draw (v7)--(v9);

\end{tikzpicture}
\end{center} 
\caption{Two graphs $G_1$ (left) and $G_2$ (right) with $B_{10}$ as a chromatic root}
\label{smallest}
\end{figure}

Of course it is easy to verify these chromatic polynomials directly by computer, but it is nonetheless interesting to 
investigate the structure of these examples and derive the result directly and in a form that can easily be
verified without using chromatic polynomial software. This is possible due to the particular convenience of chromatic polynomial
calculations (in fact, even Tutte polynomial calculations) in $2$-terminal graphs (and particularly $2$-terminal series-parallel graphs). All that is required is to maintain some auxiliary information regarding how the terminals are coloured and this information can then be propagated through the operations of series and parallel connections. 

As in Royle and Sokal \cite{MR3416855}, we use $G \ser H$ to denote the series connection of $G$ and $H$ and  $G \pll H$ for the parallel connection. We express the chromatic polynomial of a $2$-terminal graph as the sum of the $q$-colourings that colour the terminals with the {\em same} colour and those that colour the terminals with {\em different} colours. So for a 2-terminal graph $G$ with terminals $s$, $t$, we let $S_G(q)$ denote the number of $q$-colourings of $G$ where $s$ and $t$ are coloured the {\bf S}ame, and $D_G(q)$ the number of $q$-colourings of $G$ where they are coloured {\bf D}ifferently.  Then it is clear that $P_G(q) = S_G(q) + D_G(q)$ and that
\begin{align*}
S_G(q) &= P_{G/st}(q), \\
D_G(q) &= P_{G+st}(q), 
\end{align*}
where $G/st$ denotes the graph obtained by \emph{merging} $s$ and $t$ (creating a loop if $s$ and $t$ are adjacent), and $G+st$ the graph obtained by \emph{joining} $s$ and $t$. 

So now suppose that $G$ and $H$ are 2-terminal graphs and that we know $S_G$, $S_H$, $D_G$ and $D_H$. Then we can determine the  same values for the parallel connection $G \pll H$ and the series connection $G \ser H$ as follows: 
\begin{align*}
S_{G \| H} (q) &= \frac{S_G(q) S_H(q)}{q}, \\
D_{G \| H} (q) &= \frac{D_G(q) D_H(q)}{q(q-1)},\\
S_{G \ser H} (q) &= \frac{S_G (q) S_H(q)}{q} + \frac{D_G(q) D_H(q) }{ q(q-1) }, \\
D_{G \ser H} (q) &= \frac{(q-2) D_G(q) D_H(q)}{q(q-1)} + \frac{D_G(q) S_H(q)}{q} + \frac{S_G(q) D_H(q)}{q} .
\end{align*}
These expressions are calculated by taking the Cartesian product of the $q$-colourings of $G$ and $H$, dividing by some function of $q$ to ensure that we count only those pairs of colourings that match on the common vertices, and then assigning the resulting colourings to either $S$ or $D$ depending on whether the new terminals are coloured the same or different.

Now we can construct the graph $G_1$ in the following fashion:
\begin{align*}
W & =  (K_2 \ser K_2) \| (K_2 \ser K_2), \\
G_1 &= (K_6 \backslash e) \pll   (W \ser W),
\end{align*}
where $W$ is the $4$-cycle $C_4$ with nonadjacent vertices as terminals and $K_6 \backslash e$ is viewed as a 2-terminal graph by declaring the former end points of $e$ to be the two terminals.

\newcommand\ff[1]{(q)_#1}

So if we let $p_G(q) = [S_G(q), D_G(q)]$ denote the ``partitioned chromatic polynomial'' of a 2-terminal graph $G$, and let $\ff{n} = q(q-1)\cdots(q-n+1)$ denote the falling factorial (the chromatic polynomial of $K_n$), then we have
{\small
\begin{align*}
p_{K_2}(q) &=[ 0, \ff{2} ], \\
p_{K_2 \ser K_2} (q) &= [\ff{2}, \ff{3}], \\
p_{W} (q) &= [ \ff{2} (q-1), \ff{3} (q-2)] ,  \\
p_{W \ser W}  (q) &= [\ff{2} \left(q^4-7 q^3+21 q^2-29 q+15\right), \ff{3} (q-2)  \left(q^3-4 q^2+8 q-6\right)], \\
%p_x &= [\left\{( q \left\{(q-1) q \left(q^4-7 q^3+21 q^2-29 q+15\right) \left(q^4-7 q^3+21 q^2-29 q+15\right)]\\
p_{K_6\backslash e}(q) &= [\ff{5}, \ff{6}]. 
\end{align*}}
Finally, combining these last two expressions in the appropriate fashion yields the chromatic polynomial given above. A similar but slightly more complicated derivation yields the chromatic polynomial of $G_2$.

\section{Graphs with a multiple chromatic root at $q=2$}
\label{multiroot}

As outlined in \cref{intro},
Dong and Koh \cite{MR2946075} ask whether there are $3$-connected graphs without complete cutsets that have a chromatic root at $q=2$ of multiplicity greater than $2$. In this section we completely answer their question by exhibiting a family $\{X(k)\}_{k \geqslant 1}$ of such graphs with the property that $(q-2)^k \mid P_{X(k)}(q)$ for all $k \geqslant 1$.

%For all $k \geqslant 3$, the graph $X(k)$ is $3$-connected and has no complete cutset.

% \begin{figure}
% \begin{center}
% \begin{tikzpicture}[rotate=90]
% \tikzstyle{vertex}=[circle,fill=lightgray,draw=black ,inner sep = 0.5mm]

% \foreach \x/\y in {0/34,1/0,2/2,3/10,4/12,5/14, 6/22,7/24,8/26} {
%  \node [vertex] (v\x) at (10*\y:1.25cm) {};
%  %\node [vertex] (w\x) at (10*\y:2.5cm) {};
%  }
 
%  \foreach \x/\y in {0/35,1/0,2/1,3/11,4/12,5/13, 6/23,7/24,8/25} {
% % \node [vertex] (v\x) at (10*\y:1.5cm) {};
%  \node [vertex] (w\x) at (10*\y:2cm) {};
%  }

% \foreach \x in {0,1,2} {
% \foreach \y in {0,1,2} {
%   \draw (v\x)--(w\y);
%   }}
  
% \foreach \x in {3,4,5} {
% \foreach \y in {3,4,5} {
%   \draw (v\x)--(w\y);
% }}

% \foreach \x in {6,7,8} {
% \foreach \y in {6,7,8} {
%   \draw (v\x)--(w\y);
% }}

% \node [vertex] (c) at (0,0) {};
% \draw (c)--(v1);
% \draw (c)--(v4);
% \draw (c)--(v7);

% \draw [bend right = 30] (v2) to (v3);
% \draw [bend right = 30] (v5) to (v6);
% \draw [bend right = 30] (v8) to (v0);

% %\node [above] at (l2) {$L$};

% \end{tikzpicture}
% \end{center}
% \caption{$3$-connected graph with triple chromatic root at $q=2$}
% \end{figure}

The family is constructed by performing a certain \emph{vertex-replacement operation} on all the vertices of degree $3$ in a particular starting graph. This operation is best viewed as two steps. In the first step, the vertex of degree $3$ is replaced by a triangle (this is also known as a \emph{truncation} or as $Y$-$\Delta$ operation). In the second step, each edge of triangle is subdivided and an edge added between the new vertex and the single vertex of the triangle to which it is not already adjacent. The two steps together have the effect of replacing each vertex by a $6$-vertex subgraph isomorphic to $K_{3,3}$. 

The graph $X(k)$ is obtained from the $k$-spoke wheel by replacing each of the $k$ rim vertices with a copy of $K_{3,3}$ in the manner specified above. The third graph in Figure~\ref{fig:x1x2} shows $X(3)$.

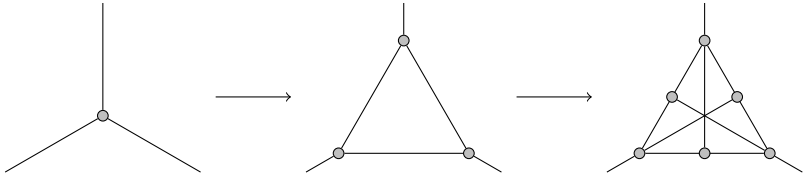
\begin{figure}
\begin{center}
\begin{tikzpicture}
\tikzstyle{vertex}=[circle,fill=lightgray,draw=black ,inner sep = 0.5mm]
\node [vertex] (c) at (0,0) {};
\foreach \x in {0,1,2} {
\coordinate (c\x) at (90+120*\x:1.5cm);
\draw (c)--(c\x);
%\node [vertex] (l\x) at (90+120*\x:1.5cm) {};
}
\draw [->] (1.5,0.25)--(2.5,0.25);
\pgftransformxshift{4cm}
\foreach \x in {0,1,2} {
\node [vertex] (l\x) at (90+120*\x:1cm) {};
\coordinate (c\x) at (90+120*\x:1.5cm);
\draw (c\x)--(l\x);
}
\draw (l0)--(l1)--(l2)--(l0);
\draw [->] (1.5,0.25)--(2.5,0.25);
\pgftransformxshift{4cm}
\foreach \x in {0,1,2} {
\node [vertex] (l\x) at (90+120*\x:1cm) {};
\coordinate (c\x) at (90+120*\x:1.5cm);
\draw (c\x)--(l\x);
}
\draw (l0)--(l1)--(l2)--(l0);
\node [vertex] (a) at ($(l0)!0.5!(l1)$) {};
\node [vertex] (b) at ($(l1)!0.5!(l2)$) {};
\node [vertex] (c) at ($(l2)!0.5!(l0)$) {};
\draw (a)--(l2);
\draw (b)--(l0);
\draw (c)--(l1);
\end{tikzpicture}
\end{center}
\caption{Replacing a vertex by a copy of $K_{3,3}$}
\label{k33replacement}
\end{figure}

To obtain the first two graphs in this family, namely $X(1)$ and $X(2)$, we must define what a $k$-spoke wheel should be when $k \in \{1,2\}$. If we take the $1$-spoke wheel to be a ``lollipop'', namely a single edge with a loop on one end-vertex, and a $2$-spoke wheel to be the graph obtained by adding a double-edge between the end-vertices of a $3$-vertex path, then the construction can be completed unambiguously and yields a well-defined simple graph in each case. Figure~\ref{fig:x1x2} shows the first three members of the family.

\begin{figure}
\begin{center}
\begin{tikzpicture}[scale = 0.8]
\tikzstyle{vertex}=[circle,fill=lightgray,draw=black ,inner sep = 0.5mm]
\node [vertex] (c) at (0,0) {};
\foreach \x in {1} {
\node [vertex] (l\x) at (15+60*\x:2.5cm) {};
\node [vertex] (m\x) at (30+60*\x:1.25cm) {};
\node [vertex] (r\x) at (45+60*\x:2.5cm) {};
\node [vertex] (a\x) at ($(l\x)!0.5!(r\x)$) {};
\node [vertex] (b\x) at ($(l\x)!0.5!(m\x)$) {};
\node [vertex] (c\x) at ($(r\x)!0.5!(m\x)$) {};

\draw (l\x)--(a\x)--(r\x);
\draw (l\x)--(b\x)--(m\x);
\draw (r\x)--(c\x)--(m\x);

\draw (a\x)--(m\x);
\draw (b\x)--(r\x);
\draw (c\x)--(l\x);

\draw (c)--(m\x);
}

%\draw (0,0) circle (2.5cm);
\draw (r1) .. controls (225:2.5cm) and (315:2.5cm) .. (l1);

% \end{tikzpicture}
% \hspace{2cm}
% \begin{tikzpicture}[rotate=-15]

\pgftransformxshift{3.75cm}
\pgftransformyshift{1cm}

\node [vertex] (c) at (0,0) {};
\foreach \x in {1,4} {
\node [vertex] (l\x) at (15+60*\x:2.5cm) {};
\node [vertex] (m\x) at (30+60*\x:1.25cm) {};
\node [vertex] (r\x) at (45+60*\x:2.5cm) {};
\node [vertex] (a\x) at ($(l\x)!0.5!(r\x)$) {};
\node [vertex] (b\x) at ($(l\x)!0.5!(m\x)$) {};
\node [vertex] (c\x) at ($(r\x)!0.5!(m\x)$) {};

\draw (l\x)--(a\x)--(r\x);
\draw (l\x)--(b\x)--(m\x);
\draw (r\x)--(c\x)--(m\x);

\draw (a\x)--(m\x);
\draw (b\x)--(r\x);
\draw (c\x)--(l\x);

\draw [black] (c)--(m\x);
}

\draw [bend left = 30] (l1) to (r4);
\draw [bend left = 30] (l4) to (r1);

\pgftransformxshift{5cm}
\pgftransformyshift{-0.25cm}

\node [vertex] (c) at (0,0) {};
\foreach \x in {1,3,5} {
\node [vertex] (l\x) at (15+60*\x:2.5cm) {};
\node [vertex] (m\x) at (30+60*\x:1.25cm) {};
\node [vertex] (r\x) at (45+60*\x:2.5cm) {};
\node [vertex] (a\x) at ($(l\x)!0.5!(r\x)$) {};
\node [vertex] (b\x) at ($(l\x)!0.5!(m\x)$) {};
\node [vertex] (c\x) at ($(r\x)!0.5!(m\x)$) {};

\draw (l\x)--(a\x)--(r\x);
\draw (l\x)--(b\x)--(m\x);
\draw (r\x)--(c\x)--(m\x);

\draw (a\x)--(m\x);
\draw (b\x)--(r\x);
\draw (c\x)--(l\x);

\draw (c)--(m\x);
}

\draw [bend left = 15] (l3) to (r1);
\draw [bend left = 15] (l5) to (r3);
\draw [bend left = 15] (l1) to (r5);

\end{tikzpicture}
\end{center}
\caption{The graphs $X(1)$, $X(2)$ and $X(3)$}
\label{fig:x1x2}
\end{figure}
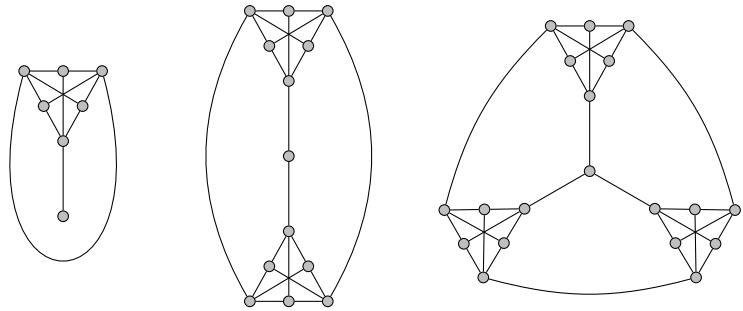

% \begin{figure}
% \begin{center}
% \begin{tikzpicture}[rotate=15]
% \tikzstyle{vertex}=[circle,fill=lightgray,draw=black ,inner sep = 0.65mm]

% \node [vertex] (c) at (0,0) {};
% \foreach \x in {0,1,...,5} {
% \node [vertex] (l\x) at (60*\x:2.5cm) {};
% \node [vertex] (m\x) at (15+60*\x:1.25cm) {};
% \node [vertex] (r\x) at (30+60*\x:2.5cm) {};
% \node [vertex] (a\x) at ($(l\x)!0.5!(r\x)$) {};
% \node [vertex] (b\x) at ($(l\x)!0.5!(m\x)$) {};
% \node [vertex] (c\x) at ($(r\x)!0.5!(m\x)$) {};

% \draw (l\x)--(a\x)--(r\x);
% \draw (l\x)--(b\x)--(m\x);
% \draw (r\x)--(c\x)--(m\x);

% \draw (a\x)--(m\x);
% \draw (b\x)--(r\x);
% \draw (c\x)--(l\x);

% \draw (c)--(m\x);
% }

% \foreach \x/\y in {0/1,1/2,2/3,3/4,4/5,5/0} {
% \draw (l\y)--(r\x);
% }
% \end{tikzpicture}
% \end{center}
% \caption{The graph $X(6)$}
% \label{fig:6spoke}
% \end{figure}

This family is a \emph{recursive family} of graphs (in the sense of Biggs, Damerell and Sands \cite{MR0294172}) which means that their Tutte polynomials $T_{X(k)}(x,y)$ satisfy a linear recurrence of the form
\[
T_{X(k+p)} + \alpha_1 T_{X(k+p-1)} + \alpha_2 T_{X(k+p-2)} + \cdots + \alpha_p T_{X(k)} = 0,
\]
where $p$ is a fixed integer, and $\alpha_i$ is a fixed polynomial in $x$ and $y$. The \emph{rank polynomial}, \emph{chromatic polynomial} and indeed any specialisation of the Tutte polynomial will satisfy an analogous linear recurrence. 

In the course of investigating which families of graphs are recursive, Noy and Rib\'o \cite{MR2037635} described a very general construction technique to produce what they called \emph{recursively constructible families} of graphs. They showed that any recursively constructible family of graphs is recursive, and conjectured the converse. The construction of the $\{X(k)\}_{k \geqslant 1}$ fits into their framework, and so is recursively constructible and therefore recursive.  In this same paper, they illustrated how to determine the linear recurrence satisfied by the rank polynomial of a recursively constructible family of graphs. 

In the current paper, we are interested only in the chromatic polynomial rather than the rank (or equivalently Tutte) polynomial. Rather than finding an expression for the rank polynomial and then specialising it to the chromatic polynomial, we exploit this fact and aim directly for the chromatic polynomial. This results in substantially smaller expressions to manipulate, and also some simplification of the overall process. However we are essentially specializing the procedure outlined by Noy and Rib\'o \cite{MR2037635} from the rank polynomial to the chromatic polynomial.

We first state our results and their consequences before devoting the remainder of this section to their proof.

\begin{thm}
\label{thm:recur}
The sequence of chromatic polynomials $P_k(q) = P_{X(k)}(q)$ of the family of graphs $\{X(k)\}_{k \geqslant 1}$ defined above satisfies the linear recurrence:
\[
P_{k+3}(q) = A(q) (q-2) P_{k+2}(q) + B(q) (q-2)^2 P_{k+1}(q) + C(q) (q-2)^3 P_k(q)
\]
where
\begin{align*}
A(q)  &= \left(q^5-9 q^4+37 q^3-88 q^2+127 q-94\right), \\
B(q)  &= \left(6 q^7-81 q^6+497 q^5-1781 q^4+4026 q^3-5780 q^2+4968 q-2024\right),
\end{align*}
\text{ and }
\begin{multline*}
C(q) = (q-1)^2  \left(3 q^2-15 q+25\right) \\ \left(3 q^5-33 q^4+157 q^3-399 q^2+535 q-299\right).
\end{multline*}
\noindent
The base cases for this recurrence are
\begin{eqnarray*}
P_1(q) & = & q (q-1)^2 (q-2) ( q^3-7 q^2+19 q-19 ),\\ 
P_2(q) & = & q (q-1) (q-2)^2 (q^9-17 q^8+138 q^7-692 q^6 \\  & & {} + 2353 q^5-5630 q^4+9525 q^3-11086 q^2+8152 q-2913),\\
P_3(q) & = &q (q-1) (q-2)^3 (q^{14}-26 q^{13}+328 q^{12}-2645 q^{11}+ 15181 q^{10} \\ 
& & {} -65498 q^9 +219032 q^8-577468 q^7 +1209533 q^6-2011958 q^5 \\ & & {} +2633017 q^4  -2650178 q^3+1957169 q^2-957460 q+235366 ). \qed
\end{eqnarray*}  

\end{thm}

\begin{cor}\label{cor:recur}
For all $k\geqslant 1$, the chromatic polynomial $P_{X(k)}$ has a factor of $(q-2)^k$. For all $k \geqslant 3$, the graph $X(k)$ is $3$-connected and has no complete cutset.
\end{cor}

\begin{proof} (Assuming \cref{thm:recur})
We prove this by induction on $k$. For the base cases, note that $P_1(q)$, $P_2(q)$ and $P_3(q)$ have chromatic polynomials with factors of $(q-2)$, $(q-2)^2$ and $(q-2)^3$ respectively. 

Now assume that $k>3$ and consider the expression for $P_k(q)$ given by the linear recurrence shown in \cref{thm:recur}. Each term in the expression has the form 
\[
Q(q) (q-2)^i P_{k-i}(q)
\]
where $i \in \{1,2,3\}$ and $Q(q)$ is one of $A(q)$, $B(q)$ and $C(q)$ respectively. By the inductive hypothesis $(q-2)^{k-i}$ is a factor of $P_{k-i}$ and so $(q-2)^k$ is a factor of each of the three terms of the sum, and so of the sum itself.

The claims about connectivity and the absence of complete cutsets are easier for the reader to verify directly by examining \cref{fig:x1x2} rather than trying to follow a tedious case analysis, which we therefore omit.
\end{proof}

The remainder of this section is devoted to the proof of \cref{thm:recur} and the details of how to derive the linear recurrence.  The members of any recursively constructible family of graphs $\{G_n\}_{n \geqslant 1}$ have a \emph{repeating structure} in that each member of the family, say $G_{k+1}$, is obtained from the previous one, $G_k$, by adding (in a sense that is made precise by Noy and Rib\'o \cite{MR2037635}) an additional copy of a fixed subgraph, which is attached to $G_k$ at certain ``vertices of attachment''. By keeping auxiliary information about the interaction between the terms of the chromatic polynomial of $G_k$ and its vertices of attachment, it is possible to determine the chromatic polynomial of $G_{k+1}$ by ``incrementally updating'' all the information pertaining to the chromatic polynomial of $G_k$. As we will see below, for a recursively constructible family of graphs, this updating takes the form of matrix multiplication by a fixed ``update matrix'' (which is usually called the \emph{transfer matrix}). Given this relatively simple relationship between $P_{X(k)}$ and $P_{X(k+1)}$ it is then possible to derive a number of recursive and/or non-recursive expressions for the chromatic polynomial of an arbitrary graph in the family. In particular, it is possible to determine the coefficients (that is, the polynomials $\alpha_1$, $\alpha_2$, etc.) of the linear recurrence satisfied by the sequence of chromatic polynomials, and it is this that we will need.

% We fill in the details and illustrate these ideas with respect to our particular family of graphs, referring the reader to Noy \& Rib\'o \cite{MR2037635} for a description of the same process in full generality.

Let $X'(k)$ denote the graph obtained from $X(k)$ by deleting one rim edge connecting distinct copies of $K_{3,3}$ (as in the first diagram of \cref{fig:5to6spoke}), and let $M$, $L$, $R$ denote the middle vertex and the left- and right-hand endpoints of the deleted edge, respectively. Then $X'(k+1)$ can be obtained by adding an additional copy of $K_{3,3}$ to the graph, and connecting it to the vertices $M$ and $R$ as shown in the second diagram in \cref{fig:5to6spoke}. In $X(k+1)$ the vertices $\{M,L,R'\}$ have the same relative positions as did $\{M,L,R\}$ in $X(k)$. Renaming $R'$ back to $R$, the ``add a $K_{3,3}$-subgraph'' step can be repeated arbitrarily many times, constructing a graph with as many copies, say $\ell$, of $K_{3,3}$ as desired.  Finally, adding the edge $LR$ to $X'(\ell)$ yields $X(\ell)$.  Our goal in the remainder of this section is to show that we can keep track of the chromatic polynomial through all of these construction steps, arriving at a suitable expression for symbolic computation and/or theoretical analysis. 

\begin{figure}
\begin{center}
\begin{tikzpicture}[rotate=15,scale=1.1]
\tikzstyle{vertex}=[circle,fill=lightgray,draw=black ,inner sep = 0.65mm]

\node [vertex] (c) at (0,0) {};
\foreach \x in {0,2,3,4,5} {
\node [vertex] (l\x) at (60*\x:2.5cm) {};
\node [vertex] (m\x) at (15+60*\x:1.25cm) {};
\node [vertex] (r\x) at (30+60*\x:2.5cm) {};
\node [vertex] (a\x) at ($(l\x)!0.5!(r\x)$) {};
\node [vertex] (b\x) at ($(l\x)!0.5!(m\x)$) {};
\node [vertex] (c\x) at ($(r\x)!0.5!(m\x)$) {};

\draw (l\x)--(a\x)--(r\x);
\draw (l\x)--(b\x)--(m\x);
\draw (r\x)--(c\x)--(m\x);

\draw (a\x)--(m\x);
\draw (b\x)--(r\x);
\draw (c\x)--(l\x);

\draw (c)--(m\x);
}

\foreach \x/\y in {2/3,3/4,4/5,5/0} {
\draw (l\y) -- (r\x);
}

%\draw [bend left = 45] (l2) to (r0);

\filldraw [thick, black, fill = white] (l2) circle (0.1cm);
\filldraw [thick, black, fill = white] (r0) circle (0.1cm);
\filldraw [thick, black, fill = white] (c) circle (0.1cm);

\node [above left] at (l2) {$L$};
\node [above right] at (r0) {$R$};
\node at (90:0.5cm) {$M$};

\end{tikzpicture}
\hspace{1.5cm}
\begin{tikzpicture}[rotate=15,scale=1.1]
\tikzstyle{vertex}=[circle,fill=lightgray,draw=black ,inner sep = 0.65mm]

\node [vertex] (c) at (0,0) {};
\foreach \x in {0,1,2,3,4,5} {
\node [vertex] (l\x) at (60*\x:2.5cm) {};
\node [vertex] (m\x) at (15+60*\x:1.25cm) {};
\node [vertex] (r\x) at (30+60*\x:2.5cm) {};
\node [vertex] (a\x) at ($(l\x)!0.5!(r\x)$) {};
\node [vertex] (b\x) at ($(l\x)!0.5!(m\x)$) {};
\node [vertex] (c\x) at ($(r\x)!0.5!(m\x)$) {};

\draw (l\x)--(a\x)--(r\x);
\draw (l\x)--(b\x)--(m\x);
\draw (r\x)--(c\x)--(m\x);

\draw (a\x)--(m\x);
\draw (b\x)--(r\x);
\draw (c\x)--(l\x);

\draw (c)--(m\x);
}

\foreach \x/\y in {0/1,2/3,3/4,4/5,5/0} {
 \draw (l\y) -- (r\x);
}

\filldraw [thick, black, fill = white] (l2) circle (0.1cm);
\filldraw [thick, black, fill = white] (r1) circle (0.1cm);
\filldraw [thick, black, fill = white] (c) circle (0.1cm);

\node [above left] at (l2) {$L$};
\node [above right] at (r0) {$R$};
\node at (105:0.5cm) {$M$};

\node [above right] at (r1) {$R'$};

\draw [thick, dashed] (75:2cm) circle (1.1cm);

%\draw [bend left = 45] (l2) to (r0);
\end{tikzpicture}
\end{center}
\caption{From $X'(5)$ to $X'(6)$}
\label{fig:5to6spoke}
\end{figure}
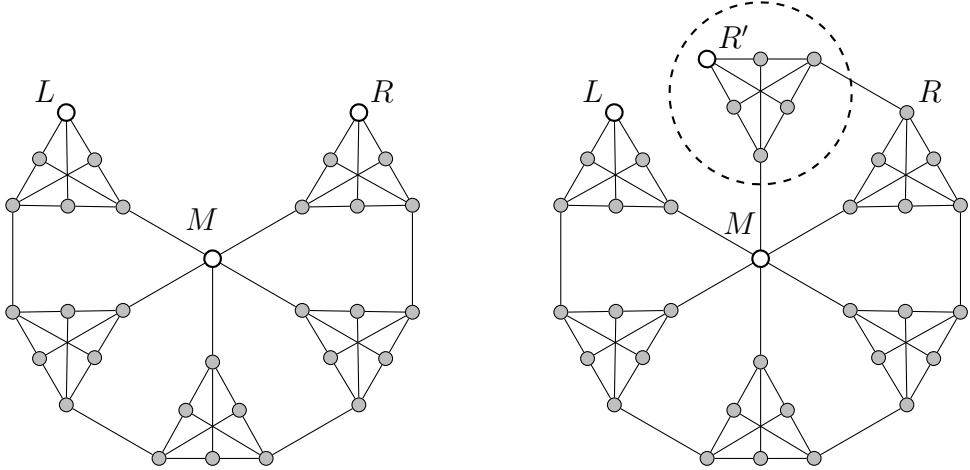

To accomplish this, we use the expansion of the chromatic polynomial of a 
graph $G = (V,E)$ as a sum over all edge-subsets, namely
\begin{equation}\label{chrompol}
P_G(q) = \sum_{A \subseteq E} (-1)^{|A|} q^{c(V,A)},
\end{equation}
where $c(V,A)$ is the number of connected components of the graph with vertex set $V$ and edge set $A$

Any subset $A$ of the edges of $X'(k)$ determines a unique partition $\pi$ of the three special vertices $\{M,L,R\}$, where two vertices are in the same cell of $\pi$ if they are in the same component of the graph $(V,A)$. Thus the chromatic polynomial of $X'(k)$ can be expressed as a sum of five ``partial'' chromatic polynomials, one for each partition of $\{M,L,R\}$. (In fact, this is just the obvious generalization of the partial chromatic polynomials for the $2$-terminal series-parallel graphs of the previous section.) Each of these five summands is obtained by restricting the chromatic polynomial summation of \cref{chrompol} to the edge-subsets inducing that particular partition. 
So if $P'$ denotes the chromatic polynomial of $X'(k)$, then 
\begin{equation}\label{rkpol}
P' = P'_{\{\mathit{MLR}\}} + P'_{\{M|\mathit{LR}\}}+P'_{\{\mathit{ML}|R\}} + P'_{\{\mathit{MR}|L\}}   + P'_{\{M|L|R\}}.
\end{equation}
% It is useful to record the contributions to $P'$ in a vector, and so we let $\boldsymbol{P}'(k)$ be the $1 \times 5$ vector whose entries are the five partial chromatic polynomials.
It is this expression of the chromatic polynomial of $X'(k)$ as a sum over partitions that comprises the auxiliary information that is maintained at each stage. (Although $L$ only appears in the final edge-addition, the auxiliary information is needed at that point and so must be maintained throughout the procedure.)

Next we consider the contribution made by each edge subset $A \subseteq E(X'(k+1))$ to the chromatic polynomial of $X'(k+1)$. Following Noy and Rib\'o \cite{MR2037635}, we set $A = B \cup C$ where $B = A \cap E(X'(k))$ and $C = A \backslash E(X'(k))$. It follows that $B \cap C = \emptyset$ and so each set of edges is the disjoint union of ``old edges'' and ``newly-added'' edges. The contribution of $A$ to the chromatic polynomial of 
$X'(k+1)$ is obtained by multiplying the contribution of $B$ (to the chromatic polynomial of $X'(k)$) by a multiplier of the form $\pm\, q^a$ for some integer $a$. The key to the entire transfer matrix method is that while the multiplier depends on $C$, and on the partition of $\{M,L,R\}$ induced by $B$, it \emph{does not depend} on $B$ itself. Moreover, the partition induced by $A$ on $\{M,L,R'\}$ depends on $C$, and on the partition of $\{M,L,R\}$ induced by $B$, but again it does not depend on $B$ itself.  

This concept is of sufficient importance that we illustrate it with a particular example. \cref{specialC} shows a particular $6$-edge subset of the new edges. Suppose then that $B$ is a subset of the old edges about which we know nothing other than it contains a path connecting $L$ to $M$, but not from either of those two vertices to $R$ (this is shown schematically as squiggly lines passing from $L$ to $M$ in the rest of the graph). In other words, it induces the partition $\pi = \{\mathit{ML}|R\}$ on $\{M,L,R\}$. Forming $A = B \cup C$ by the addition of $C$ results in six additional edges, just one additional connected component (the isolated vertex in the $K_{3,3}$), and the resulting set $A$ contains paths between any two vertices in $\{L,M,R'\}$. The total contribution to the chromatic polynomial of $X'(k)$ from all of the edge-subsets of $E(X'(k))$ that induce the partition $\pi$ is $P'_{\{\mathit{ML}|R\}}$. Augmenting each of these edge-subsets by $C$ multiplies each of the contributions by $(-1)^6 q = q$, and so $q P'_{\{\mathit{ML}|R\}}$ is the total contribution from all such edge-subsets. In each case, the partition induced on $\{M,L,R'\}$ by the edge-subset augmented by $C$ is the same, namely $\{MLR'\}$, and so the resulting contribution can be allocated in its entirety to the relevant partial chromatic polynomial for $X'(k+1)$. 

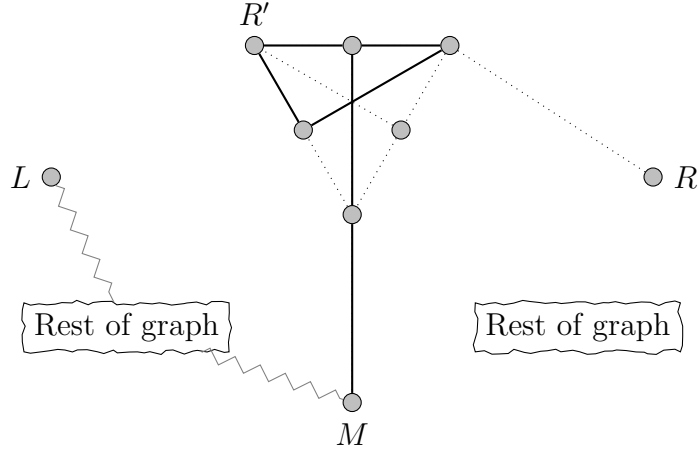
\begin{figure}
\begin{center}
\begin{tikzpicture}
\tikzstyle{vertex}=[circle,fill=lightgray,draw=black ,inner sep = 0.85mm]
\node [vertex] (v0) at (270:1.5cm) {};
\node [vertex] (v1) at (30:1.5cm) {};
\node [vertex,label=90:{$R'$}] (v2) at (150:1.5cm) {};
\node [vertex, label=270:{$M$}] (vM) at (0,-4) {};
\node [vertex,label=180:{$L$}] (vL) at (-4,-1) {};
\node [vertex,label=0:{$R$}] (vR) at (4,-1) {};
\node [vertex] (a2) at ($(v0)!0.5!(v1)$) {};
\node [vertex] (a0) at ($(v1)!0.5!(v2)$) {};
\node [vertex] (a1) at ($(v2)!0.5!(v0)$) {};
\draw [dotted] (v0)--(a1)--(v2)--(a0)--(v1)--(a2)--(v0);
\draw [dotted] (a0)--(v0);
\draw [dotted] (a1)--(v1);
\draw [dotted] (a2)--(v2);
\draw [dotted] (v0)--(vM);
\draw [dotted] (v1)--(vR);
%\draw [dotted] (v2)--(vL);
\node [draw, decorate, decoration={random steps,segment length=3pt,amplitude=1pt}] (rog) at (-3,-3)
{Rest of graph};
\node [draw, decorate, decoration={random steps,segment length=3pt,amplitude=1pt}] (rog2) at (3,-3)
{Rest of graph};
\draw [gray, decorate, decoration=zigzag](vL)--(rog);
\draw [gray, decorate, decoration=zigzag](rog)--(vM);
\draw [thick] (v0)--(vM);
%\draw [thick] (v1)--(a0);
\draw [thick] (v2)--(a1);
\draw [thick] (a1)--(v1);
\draw [thick] (v1)--(a0);
\draw [thick] (a0)--(v2);
\draw [thick] (a0)--(v0);
%\draw [thick] (a1)--(v0);
%\draw [thick] (v1)--(vR);
%\draw [thick] (v1)--(a1);

\end{tikzpicture}
\end{center}
\caption{An edge set contributing to the $\left( \{\mathit{ML}|R\},\{\mathit{MLR}\}\right)$-entry of $T$}
\label{specialC}
\end{figure}

The transfer matrix $T$ is obtained by considering all possible subsets of the new edges, and accumulates the multipliers associated with each subset. It has rows indexed by the five partitions of $\{M,L,R\}$ and columns indexed by the five partitions of $\{M,L,R'\}$; both use the same ordering of partitions as in \cref{rkpol}.
Its entries are determined as follows: for each subset $C$ of the new edges, and each possible partition $\pi$ of $\{M,L,R\}$, we compute the multiplier $\pm\, q^a$ and the partition $\pi'$ of $\{M,L,R'\}$ determined by $C$ and $\pi$. This multiplier is added to the $(\pi,\pi')$-entry of $T$.  There are just $2^{11}$ choices for $C$ and five choices for $\pi$, and the only computation required for each combination is to count connected components and decide which vertices are in the same component in a $9$-vertex graph. Therefore it is straightforward to determine the transfer matrix, which is given by:

\[
T = 
\left[\begin{array}{ccccc}
a_1 & 0 & 0 & a_2 & 0 \\
a_3 & a_4 & a_4 & a_5 & a_6 \\
0 & 0 & a_1 & 0 & a_2 \\
a_7 & 0 & 0 & a_8 & 0 \\
0 & 0 & a_7 & 0 & a_8
\end{array}\right]
\]
where 
\begin{align*}
a_1 &= -6 \, q^{3} + 39 \, q^{2} - 89 \, q + 69 ,\\
a_2 &=q^{6} - 11 \, q^{5} + 55 \, q^{4} - 147 \, q^{3} + 204 \, q^{2} - 115 \, q ,\\
a_3 &=-3 \, q - \frac{31}{q} + 21 ,\\
a_4 &=-3 \, q^{3} + 21 \, q^{2} - 55 \, q + 50 ,\\
a_5 &=3 \, q^{2} - 15 \, q + 19 ,\\
a_6 &=q^{6} - 11 \, q^{5} + 55 \, q^{4} - 150 \, q^{3} + 219 \, q^{2} - 134 \, q ,\\
a_7 &=-3 \, q^{3} + 21 \, q^{2} - 58 \, q - \frac{31}{q} + 71 ,\\
a_8 &=q^{6} - 11 \, q^{5} + 55 \, q^{4} - 153 \, q^{3} + 243 \, q^{2} - 204 \, q + 69.
\end{align*}

If we let $\boldsymbol{P}'(k)$ be the row vector of length $5$, whose entries are the five partial chromatic polynomials of $X(k)$ as in \cref{rkpol}, then 
\[
\boldsymbol{P}'(k+1) = \boldsymbol{P}'(k) T.
\] 
This follows
because the matrix multiplication encapsulates the bookkeeping necessary to multiply each of the partial chromatic polynomials of $X'(k)$ by its appropriate multiplier, and accumulate the results in the correct partial chromatic polynomials of $X'(k+1)$. 

Two things remain before we can extract a complete expression for the chromatic polynomial of $X(k)$, namely a base case for the recursion, and an adjustment to account for the final step where the edge $LR$ is added to ``close the circuit'' and construct $X(k)$ from $X'(k)$.

For the base case, we need the vector $\boldsymbol{P}'(1)$ which is the chromatic polynomial of $X'(1)$ (the first graph of \cref{fig:x1x2} with the curved edge deleted) with all the terms separated into the five partial chromatic polynomials as before. As $X'(1)$ has only ten edges, a straightforward computation over the $2^{10}$ edge-subsets yields:
\[
{{\boldsymbol{P'}(1)}}^T = 
\left[
\begin{array}{r}
3 \, q^{3} - 21 \, q^{2} + 31 \, q\\
3 \, q^{5} - 21 \, q^{4} + 55 \, q^{3} - 50 \, q^{2}\\
-3 \, q^{4} + 15 \, q^{3} - 19 \, q^{2}\\
-3 \, q^{4} + 15 \, q^{3} - 19 \, q^{2}\\
q^{7} - 10 \, q^{6} + 42 \, q^{5} - 84 \, q^{4} + 65 \, q^{3}\\
\end{array}
\right].
\]
(Note that ${\boldsymbol{P'}(1)}$ is a row-vector, but for typographical reasons we show its transpose.)  If $\boldsymbol{j}$ is the all-ones (column) vector of length $5$, then we now have
\[
P_{X'(k)} = \boldsymbol{P'}(1)\, T^{k-1} \, \boldsymbol{j},
\]
because post-multiplying by $\boldsymbol{j}$ simply adds up the five partial chromatic polynomials in $\boldsymbol{P'}(k)$.

Now we consider adding the final edge $LR$, thereby constructing $X(k)$ from $X'(k)$. Let $B \subseteq E(X'(k))$, and then let $C = B \cup \{LR\}$. If $B$ contains edges connecting $L$ to $R$, then $c(V,B)  = c(V,C)$, otherwise $c(V,C) = c(V,B) -1$. In the first case the contributions of $B$ and $C$ to the summation cancel each other out because $|C| = |B|+1$. In the second case, $C$ contributes $-1/q$ times the contribution of $B$. Thus in total, the chromatic polynomial of $X(k)$ is the sum of the partial chromatic polynomials of $X'(k)$ associated with the partitions where $L$ and $R$ are in different cells, all multiplied by $1-1/q$.
So if we put
\[
\boldsymbol{v}^T = (1-1/q) \cdot \left( 0, 0, 1, 1, 1 \right)
\]
then 
\[
P_{X(k)} = \boldsymbol{P'}(1)\, T^{k-1} \, \boldsymbol{v}.
\]

Still following Noy and Rib\'o \cite{MR2037635}, the linear recurrence satisfied by the sequence $P_{X(k)}$ can be extracted by considering the \emph{generating function} 
\[
\sum_{k=1} P_{X(k)} z^k = \boldsymbol{P}'(1) (I - z T)^{-1} \boldsymbol{v}
\]
and then expressing this as a rational function; the denominator of this rational function is a cubic polynomial in $z$, and its coefficients are the coefficients of the linear recurrence satisfied by the sequence of chromatic polynomials.

Having obtained the linear recurrence, it is then possible to verify  that it is correct. This follows because it is easy to check that  
\[
 T^3 - A(q) (q-2) T^2 - B(q) (q-2)^2 T - C(q) (q-2)^3  = 0
\]
and pre-multiplying this expression by $\boldsymbol{P}'(1)T^{k-1}$, post-multiplying it by $\boldsymbol{v}$ and then expanding out leaves the desired recurrence.

This concludes the derivation of the linear recurrence, completing the proof of \cref{thm:recur}.

\bibliographystyle{plain}
\bibliography{sample.bib}

% For the base cases for the recurrence, we consider the graphs $X(1)$ and $X(2)$. (Although a $1$-spoke wheel takes some interpretation, it can be done.)

%% The Appendices part is started with the command \appendix;
%% appendix sections are then done as normal sections
\appendix

\section{Code}

This Appendix contains code in {\tt SageMath} (which is a mild variation on Python) that will enable the reader to rapidly construct any member of the family of graphs $\{X(k)\}$, and verify the correctness of the claimed chromatic polynomials. 

{\tt SageMath} (\url{sagemath.org}) is a free open-source computer algebra package that is readily available, either by downloading and installing the software on a local machine, or by using a web-browser to access a free cloud-based interface.

There are three {\tt SageMath} functions, firstly {\tt Xkgraph(k)} which constructs the graph $X(k)$, then {\tt nextPol(pA,pB,pC)} which takes as arguments the polynomials $P_{k-1}$, $P_{k-2}$ and $P_{k-3}$ and returns $P_k$. 
These functions are used by teh hel

These functions are called by {\tt cpols(k)} which returns an array containing the chromatic polynomials of the first $k$ graphs in the sequence. This is only function that needs to be called by the user.

\begin{minted}{python}
def Xkgraph(k):
    nv = 6*k+1
    g = Graph(nv)
    for x in range(k):
        for i in range(3):
            for j in range(3):
                g.add_edge(6*x+1+i,6*x+4+j)
        g.add_edge(0,6*x+1)
        g.add_edge(6*x+3,6*((x+1)%k)+2)       
    return g
\end{minted}

\begin{minted}{python}
def nextPol(pA, pB, pC):
    A=q^5-9*q^4+37*q^3-88*q^2+127*q-94
    B=6*q^7-81*q^6+497*q^5-1781*q^4+
    	4026*q^3-5780*q^2+4968*q-2024
    C=(q-1)^2*(3*q^2-15*q+25)*(3*q^5-33*q^4+
    	157*q^3-399*q^2+535*q-299)
    ppol = A*(q-2)*pA+B*(q-2)^2*pB +C*(q-2)^3*pC
    return ppol.factor()
\end{minted}

\begin{minted}{python}
def cpols(k):
    q = var('q')
    gs = [Xkgraph(i) for i in range(1,4)]
    cps = [g.chromatic_polynomial().subs(x=q) for g in gs]

    for i in range(3,k):
        cps.append(nextPol(cps[i-1],cps[i-2],cps[i-3]))
        
    return cps
\end{minted}

\end{document}